\documentclass{article}
\usepackage[utf8]{inputenc}
\usepackage{url}
\usepackage[colorlinks,allcolors=black]{hyperref}
\hypersetup{
     colorlinks = true,
     citecolor = black
}
\usepackage[utf8]{inputenc}
\usepackage{amssymb}
\usepackage{mathrsfs}
\usepackage{amsthm}
\usepackage{amsmath}
\usepackage{hyperref}
\usepackage{cleveref}
\usepackage[french,USenglish]{babel}
\usepackage{float}
\usepackage{color}
\usepackage{graphicx}
\usepackage[explicit]{titlesec}
\usepackage{sectsty}
\usepackage{pict2e}
\usepackage[T1]{fontenc}
\usepackage{filecontents}
\usepackage{mathtools}
\usepackage{lipsum}
\usepackage{dsfont}
\usepackage{stmaryrd}
\usepackage{setspace}
\usepackage{tikz}
\usepackage{subfig}

\newtheorem{theorem}{Theorem}

\newtheorem{lemma}[theorem]{Lemma}

\newtheorem{remark}[theorem]{Remark}

\newcommand{\E}{{\mathbb E}}

\DeclareMathOperator*{\esssup}{ess\,sup}

\usepackage{geometry}
\newcommand{\N}{\mathbb{N}}
\newcommand{\R}{\mathbb{R}}

\title{Maximal Martingale Wasserstein Inequality}

\author{Benjamin Jourdain\thanks{CERMICS, Ecole des Ponts, INRIA, Marne-la-Vallée, France. E-mail: benjamin.jourdain@enpc.fr - This research benefited from the support of the \textquotedblleft Chaire Risques Financiers\textquotedblright , Fondation du Risque.}
\and 
Kexin Shao\thanks{
INRIA Paris,  2 rue Simone Iff, CS 42112, 75589 Paris Cedex 12, France, Universit\'e Paris-Dauphine, Ecole des Ponts ParisTech. E-mail: kexin.shao@inria.fr. This project has received funding from the European Union’s Horizon 2020 research and innovation
programme under the Marie Sk\l{}odowska-Curie grant agreement No 945322.
}}
\date{\today}

\begin{document}
\maketitle
\begin{abstract}
In this note, we complete the analysis of the Martingale Wasserstein Inequality started in \cite{JoMa22} by checking that this inequality fails in dimension $d\ge 2$ when the integrability parameter $\rho$ belongs to $[1,2)$ while a stronger Maximal Martingale Wasserstein Inequality holds whatever the dimension $d$ when $\rho\ge 2$.
\end{abstract}

\section{Introduction}
The present paper elaborates on the convergence to $0$ as $n\to\infty$ of $\inf_{M\in\Pi^{\mathrm{M}}(\mu_n,\nu_n)}\int_{\R^d\times\R^d}|y-x|^\rho M(dx,dy)$ with the Wasserstein distance $\mathcal W_\rho(\mu_n,\nu_n)$ when for each $n\in\N$, $\mu_n$ and $\nu_n$ belong to the set $\mathcal P_\rho(\R^d)$ of probability measures on $\R^d$ with a finite moment of order $\rho\in[1,+\infty)$ and the former is smaller than the latter in the convex order. The convex order between $\mu,\nu\in\mathcal P_1(\R^d)$ which is denoted $\mu\le_{cx}\nu$ amounts to
\begin{equation}\label{def:convexOrder}
\int_{\R^d}f(x)\,\mu(dx)\le\int_{\R^d}f(y)\,\nu(dy)\mbox{ for each convex function }f:\R^d\to\R,
\end{equation}
and, by Strassen's theorem \cite{St65}, is equivalent to the non emptyness of the set of martingale couplings between $\mu$ and $\nu$ defined by 
\begin{align*}
   &\Pi^{\mathrm{M}}(\mu,\nu)=\left\{M(dx,dy)=\mu(dx)m(x,dy)\in\Pi(\mu,\nu)\mid\mu(dx)\text{-a.e.},\ \int_{\R^d}y\,m(x,dy)=x\right\}\mbox{ where }\\
&\Pi(\mu,\nu)=\{\pi\in\mathcal P_1(\R^d \times \R^d)\mid\forall A\in\mathcal B(\R^d),\ \pi(A\times\R^d)=\mu(A)\ \mathrm{and}\ \pi(\R^d\times A)=\nu(A)\}.
\end{align*}
The Wasserstein distance with index $\rho$ is defined by
\[
\mathcal W_\rho(\mu,\nu)=\left(\inf_{\pi\in\Pi(\mu,\nu)}\int_{\R^d\times\R^d}\vert x-y\vert^\rho\,\pi(dx,dy)\right)^{1/\rho}
\]
and we also introduce $\underline{\mathcal M}_\rho(\mu,\nu)$ and $\overline{\mathcal M}_\rho(\mu,\nu)$ respectively defined by 
\begin{align}
  \underline{\mathcal M}^\rho_\rho(\mu,\nu)=\inf_{M\in\Pi^{\mathrm{M}}(\mu,\nu)}\int_{\R^{2d}}\vert x-y\vert^\rho M(dx,dy),\;\overline{\mathcal M}^\rho_\rho(\mu,\nu)=\sup_{M\in\Pi^{\mathrm{M}}(\mu,\nu)}\int_{\R^{2d}}\vert x-y\vert^\rho\,M(dx,dy).
\end{align}
In dimension $d=1$, the optimization problems defining $\underline{\mathcal M}_\rho$ and $\overline{\mathcal M}_\rho$ are the respective subjects of \cite{HoKl} and \cite{HobsonNeuberger} when $\rho=1$, while the general case $\rho\in (0,+\infty)$ is studied in \cite{JoSh23}.

The question of interest is related to the stability of Martingale Optimal Transport problems with respect to the marginal distributions $\mu$ and $\nu$ established in dimension $d = 1$ in \cite{BaPa22, Wi20} while it fails in higher dimension according to \cite{BruJu22}. A quantitative answer is given in dimension $d=1$ by the Martingale Wasserstein inequality established in \cite[Proposition 1]{JoMa22} for $\rho\in[1,+\infty)$,
\begin{align}
   \exists \underline{C}_{(\rho,\rho),1}<\infty,\;\forall \mu,\nu\in{\cal P}_\rho(\R)\mbox{ with }\mu\le_{cx}\nu,\;\underline{\mathcal M}^\rho_\rho(\mu,\nu)\le \underline{C}_{(\rho,\rho),1} {\cal W}_\rho(\mu,\nu)\sigma_\rho^{\rho -1}(\nu),\label{inmartwass}
\end{align}
where the central moment $\sigma_\rho(\nu)$ of $\nu$ is defined by
\[\sigma_\rho(\nu)=\inf_{c\in\R^d}\left(\int_{\R^d}\vert y-c\vert^\rho\,\nu(dy)\right)^{1/\rho}\mbox{ when }\rho\in[1,+\infty)\mbox{ and }\sigma_\infty(\nu)=\inf_{c\in\R^d}\nu-\esssup_{y\in\R^d}|y-c|.
\]
The proposition also states that ${\cal W}_\rho(\mu,\nu)$ and $\sigma_\rho(\nu)$ have the right exponent in this inequality in the sense that for $1<s<\rho$, $\sup_{\stackrel{\mu,\nu\in{\cal P}_\rho(\R)}{\mu\le_{cx}\nu,\mu\neq \nu}}\frac{\underline{\mathcal M}^\rho_\rho(\mu,\nu)}{{\cal W}^s_\rho(\mu,\nu)\sigma_\rho^{\rho -s}(\nu)}=+\infty$.
The generalization of \eqref{inmartwass} to higher dimensions $d$ is also investigated in \cite{JoMa22} where it is proved that for any $d\ge 2$, $$\underline{C}_{(\rho,\rho),d}:=\sup_{\stackrel{\mu,\nu\in{\cal P}_\rho(\R^d)}{\mu\le_{cx}\nu,\mu\neq \nu}}\frac{\underline{\mathcal M}^\rho_\rho(\mu,\nu)}{{\cal W}_\rho(\mu,\nu)\sigma_\rho^{\rho -1}(\nu)}$$
is infinite when $\rho\in [1,\frac{1+\sqrt{5}}2)$
, while the one-dimensional constant $\underline{C}_{(\rho,\rho),1}$ is preserved when $\mu$ and $\nu$ are products of one-dimensional probability measures or when, for $X$ distributed according to $\mu$, the conditional expectation of $X$ given the direction of $X-\E[X]$ is a.s. equal to $\E[X]$ and $\nu$ is the distribution of $X+\lambda(X-\E[X])$ for some $\lambda\ge 0$.
The present paper answers the question of the finiteness of $\underline{C}_{(\rho,\rho),d}$ when $\rho\in[\frac{1+\sqrt{5}}2,+\infty)$ and $d\ge 2$, which remained open. It turns out that 
$\underline{C}_{(\rho,\rho),d}=+\infty$ for $d\ge 2$ when $\rho\in[1,2)$ while for $\rho\in[2,+\infty)$ the inequality \eqref{inmartwass} generalizes in any dimension $d$ into a Maximal Martingale Wasserstein inequality with the left-hand side $\underline{\mathcal M}^\rho_\rho(\mu,\nu)$ replaced by the larger $\overline{\mathcal M}^\rho_\rho(\mu,\nu)$. We even replace conjugate exponents $\rho$ and $\frac{\rho}{\rho-1}$ leading to the respective indices $\rho=\rho\times 1$ and $\rho=\frac{\rho}{\rho-1}\times(\rho-1)$ in the factors ${\cal W}$ and $\sigma$ in \eqref{inmartwass} by general conjugate exponents $q\in[1,+\infty]$ and $\frac{q}{q-1}\in[1,+\infty]$ leading to indices $q$ and $\frac{q(\rho-1)}{q-1}$ (equal to $+\infty$ and $\rho-1$ when $q$ is respectively equal to $1$ and $+\infty$)
and define $$\underline{C}_{(\rho,q),d}:=\sup_{\stackrel{\mu,\nu\in{\cal P}_{q\vee\frac{(\rho-1)q}{q-1}}(\R^d)}{\mu\le_{cx}\nu,\mu\neq \nu}}\frac{\underline{\mathcal M}_\rho^\rho(\mu,\nu)}{\mathcal W_{q}(\mu,\nu)\sigma_{\frac{q(\rho-1)}{q-1}}^{\rho-1}(\nu)}\mbox{ and }\overline{C}_{(\rho,q),d}:=\sup_{\stackrel{\mu,\nu\in{\cal P}_{q\vee\frac{(\rho-1)q}{q-1}}(\R^d)}{\mu\le_{cx}\nu,\mu\neq \nu}}\frac{\overline{\mathcal M}_\rho^\rho(\mu,\nu)}{\mathcal W_{q}(\mu,\nu)\sigma_{\frac{q(\rho-1)}{q-1}}^{\rho-1}(\nu)},$$
with $\mathcal W_\infty(\mu,\nu)=\inf_{\pi\in\Pi(\mu,\nu)}\pi-\esssup_{(x,y)\in\R^d\times\R^d}\vert x-y\vert$.
Since $\underline{\mathcal M}_\rho\le \overline{\mathcal M}_\rho$, one has $\underline{C}_{(\rho,q),d}\le \overline{C}_{(\rho,q),d}$. These constants of course depend on the norm $|\cdot|$ on $\R^d$ (even if we do not make this dependence explicit) but, by equivalence of the norms, their finiteness does not. Since the Euclidean norm plays a particular role, we will denote it by $\|\cdot\|$ rather than $|\cdot|$.
\begin{theorem}\label{thmconst}
  \begin{description}
  \item[$(i)$]Let $\rho\in[1,2)$. For $q\in[1,\frac{1}{2-\rho}]$ (and even $q\in[1,+\infty]$ when $\rho=1$), one has $\underline{C}_{(\rho,q),1}\le K_\rho<+\infty$ where the constant $K_\rho$ is studied in \cite[Proposition 1]{JoMa22} while, for $q\in[1,+\infty]$, $\overline{C}_{(\rho,q),1}=+\infty$ and $\underline{C}_{(\rho,q),d}=+\infty$ for $d\ge 2$.
\item[$(ii)$] Let $\rho\in [2,+\infty)$ and $q\in[1,+\infty]$. One has $\overline{C}_{(\rho,q),d}<+\infty$ whatever $d$. Moreover, when $\R^d$ (resp. each $\R^d$) is endowed with the Euclidean norm, $\overline{C}_{(2,q),d}=2$ and $\sup_{d\ge 1}\overline{C}_{(\rho,q),d}<+\infty$.
  
  \end{description} 
\end{theorem}

\begin{remark}
   \begin{itemize}
   \item The fact that $\rho=2$ appears as a threshold is related to the equality $\int_{\R^d\times \R^d}\|y-x\|^2M(dx,dy)=\int_{\R^d}\|y\|^2\nu(dy)-\int_{\R^d}\|x\|^2\mu(dx)$ for $M\in\Pi^M(\mu,\nu)$ when $\mu,\nu\in{\cal P}_2(\R^d)$ are such that $\mu\le_{cx}\nu$, which implies that when $\R^d$ is endowed with the Euclidean norm $$\underline{\mathcal M}^2_2(\mu,\nu)=\overline{\mathcal M}^2_2(\mu,\nu)=\int_{\R^d}\|y\|^2\nu(dy)-\int_{\R^d}\|x\|^2\mu(dx).$$
     \item For $\rho\in[1,2)$, one has $\overline{C}_{(\rho,q),d}=+\infty$ while $\sup_{\stackrel{\mu,\nu\in{\cal P}_{q\vee\frac{q}{q-1}}(\R^d)}{\mu\le_{cx}\nu,\mu\neq \nu}}\frac{\overline{\mathcal M}_\rho^2(\mu,\nu)}{\mathcal W_{q}(\mu,\nu)\sigma_{\frac{q}{q-1}}(\nu)}\le \overline{C}_{(2,q),d}<+\infty$ since $\overline{\mathcal M}_\rho\le\overline{\mathcal M}_2$.
   \end{itemize}
\end{remark}

\section{Proof}
The proof of Theorem \ref{thmconst} $(ii)$ relies on the next lemma, the proof of the lemma is postponed after the proof of the theorem. In what follows, to avoid making distinctions in case $q\in\{1,+\infty\}$, we use the convention that for any probability measure $\gamma$ and any measurable function $f$ on the same probability space $\left(\int |f(z)|^q\gamma(dz)\right)^{1/q}$ (resp. $\left(\left(\int |f(z)|^{\frac{q}{q-1}}\gamma(dz)\right)^{(q-1)/q},\left(\int |f(z)|^{\frac{q(\rho-1)}{q-1}}\gamma(dz)\right)^{(q-1)/q}\right)$) is equal to $\gamma-\esssup_{z}|f(z)|$ (resp. $(\gamma-\esssup_{z}|f(z)|,\gamma-\esssup_{z}|f(z)|^{\rho-1})$) when $q=+\infty$ (resp. $q=1$).
\begin{lemma}\label{c1c2lemma}
Given $\rho\in[2,+\infty)$, there exist constants $\kappa_{\rho}, \tilde \kappa_{\rho} \in[0,+\infty)$ such that for all $d\ge 1$ and $x,y\in\R^d$,
\begin{equation}\label{c1lemma}
{\lVert x-y \rVert}^{\rho} \leq \kappa_\rho \left((\rho-1){\lVert x \rVert}^{\rho} + {\lVert y \rVert}^{\rho} -\rho\lVert x\rVert^{\rho-2}\langle x, y \rangle\right),
\end{equation}

\begin{equation}\label{c2lemma}
{\lVert y \rVert}^{\rho} - {\lVert x \rVert}^{\rho} \leq {\tilde \kappa}_\rho {\lVert y-x \rVert} \left( {\lVert x \rVert}^{\rho-1} +{\lVert y \rVert}^{\rho-1} \right).
\end{equation}
\end{lemma}
\begin{remark}\label{remrho2}
  When $\rho=2$, then \eqref{c1lemma} holds as an equality with $\kappa_\rho=1$ while, by the Cauchy-Schwarz and the triangle inequalities,
\begin{equation*}
    \lVert y \rVert^2 - \lVert x \rVert^2 \leq \langle y-x, y+x \rangle \leq \lVert y-x \rVert\times \lVert y+x\rVert 
    \leq \lVert y-x \rVert \left(\lVert x \rVert +\lVert y \rVert \right)
\end{equation*}
so that \eqref{c2lemma} holds with ${\tilde \kappa}_\rho = 1$.
\end{remark}

\begin{proof}[Proof of Theorem \ref{thmconst}]
  $(i) $ In dimension $d=1$, one has $\underline{\mathcal M}_1\le K_1 {\mathcal W}_1$ with $K_1=2$ according to
\cite[Proposition 1]{JoMa22}
and we deduce that $\underline{C}_{(1,q),1}\le K_1$ for $q\in[1,+\infty]$ since ${\mathcal W}_1\le {\mathcal W}_q$. Let now $\rho\in(1,2)$ and $q\in[1,\frac{1}{2-\rho}]$. One has $\frac{q(\rho-1)}{q-1}\ge 1$ since, when $q>1$, $\frac{q}{q-1}=1+\frac 1{q-1}\ge 1+\frac{2-\rho}{\rho-1}=\frac 1{\rho -1}$. For $\mu,\nu\in{\cal P}_{q\vee \frac{q(\rho-1)}{q-1}}(\R)$ with respective quantile functions $F_\mu^{-1}$ and $F_\nu^{-1}$, one has by optimality of the comonotonic coupling and H\"older's inequality
  \begin{align*}
   {\cal W}_\rho^\rho(\mu,\nu)&=\int_0^1|F_\nu^{-1}(u)-F_\mu^{-1}(u)|\times |F_\nu^{-1}(u)-F_\mu^{-1}(u)|^{\rho-1}du\\&\le \left(\int_0^1|F_\nu^{-1}(u)-F_\mu^{-1}(u)|^qdu\right)^{1/q}\left(\left(\int_0^1|F_\nu^{-1}(u)-F_\mu^{-1}(u)|^{\frac{q(\rho-1)}{q-1}}du\right)^{\frac{q-1}{q(\rho-1)}}\right)^{\rho-1}.
  \end{align*}
  Since, by the triangle inequality and $\mu\le_{cx}\nu$, one has for $c\in\R$
 \begin{align*}
   \left(\int_0^1|F_\nu^{-1}(u)-F_\mu^{-1}(u)|^{\frac{q(\rho-1)}{q-1}}du\right)^{\frac{q-1}{q(\rho-1)}}&\le \left(\int_0^1|F_\nu^{-1}(u)-c|^{\frac{q(\rho-1)}{q-1}}du\right)^{\frac{q-1}{q(\rho-1)}}+ \left(\int_0^1|F_\mu^{-1}(u)-c|^{\frac{q(\rho-1)}{q-1}}du\right)^{\frac{q-1}{q(\rho-1)}}\\
  &\le 2\left(\int_0^1|F_\nu^{-1}(u)-c|^{\frac{q(\rho-1)}{q-1}}du\right)^{\frac{q-1}{q(\rho-1)}},
 \end{align*}
 we deduce by minimizing over the constant $c$ that
\begin{align*}
   {\cal W}_\rho^\rho(\mu,\nu)\le {\cal W}_q(\mu,\nu)\times 2^{\rho-1}\sigma^{\rho-1}_{\frac{q(\rho-1)}{q-1}}(\nu).
\end{align*}
With this inequality replacing (30) in the proof of Proposition 1 \cite{JoMa22} and the general inequality
$$\int_0^1|F_\nu^{-1}(u)-F_\mu^{-1}(u)||F_\nu^{-1}(u)-c|^{\rho-1}du\le {\cal W}_q(\mu,\nu)\left(\int_0^1|F_\nu^{-1}(u)-c|^{\frac{q(\rho-1)}{q-1}}du\right)^{\frac{q-1}{q}},$$
replacing the special case $q=\rho$ in the second equation p840 in this proof, we deduce that ${\cal W}_\rho^\rho(\mu,\nu)\le K_\rho {\cal W}_q(\mu,\nu)\sigma^{\rho-1}_{\frac{q(\rho-1)}{q-1}}(\nu)$.

To check that $\overline{C}_{(\rho,q),1}=+\infty$ for $\rho\in[1,+\infty)$ and $q\in[1,+\infty]$, let us introduce for $n\ge 2$ and $z>0$, 
\begin{align*}
   &\mu_{n,z}=\frac{1}{2((n-1)z+1)}\left((1+z)\left(\delta_1+\delta_n\right)+2z\sum_{i=2}^{n-1}\delta_i\right)\\\mbox{ and }&\nu_{n,z}=\frac{1}{2((n-1)z+1)}\left(\delta_{1-z}+\delta_{n+z}+z\left(\delta_1+\delta_n\right)+2z\sum_{i=2}^{n-1}\delta_i\right).
\end{align*} 
This example generalizes the one introduced by Br\"uckerhoff and Juillet in \cite{BruJu22} which corresponds to the choice $z=1$. Since $$M_{n,z}=\frac{1}{2((n-1)z+1)}\left(\delta_{(1,1-z)}+z\delta_{(1,2)}+z\delta_{(n,n-1)}+\delta_{(n,n+z)}+z\sum_{i=2}^{n-1}\left(\delta_{(i,i-1)}+\delta_{(i,i+1)}\right)\right)$$
belongs to $\Pi^M(\mu_{n,z},\nu_{n,z})$, we have 
\begin{align*}
   \overline{\mathcal{M}}^\rho_\rho (\mu_{n,z}, \nu_{n,z})\ge \int_{\R\times\R}|y-x|^\rho M_{n,z}(dx,dy)=\frac{(n-1)z+z^\rho}{(n-1)z+1}.
\end{align*}
On the other hand, by optimality of the comonotonic coupling $\mathcal{W}^\rho_\rho (\mu_{n,z}, \nu_{n,z}) =\frac{z^\rho}{(n-1)z+1}$ for $\rho\in[1,+\infty)$ and $\mathcal{W}_\infty(\mu_{n,z}, \nu_{n,z})=z$. Last $\sigma_\infty(\nu_{n,z})=\frac{n-1+2z}{2}$ and, when $\rho\in[1,+\infty)$,
\begin{align*}
   \sigma_\rho^\rho(\nu_{n,z})=\frac{1}{2^\rho((n-1)z+1)}\left((n-1+2z)^\rho+z(n-1)^\rho+2z\sum_{i=2}^{\lfloor \frac{n+1}2\rfloor}(n+1-2i)^\rho\right).
\end{align*}

Let $\alpha\in [0,1)$. The sequence $n^{1-\alpha}$ goes to $\infty$ with $n$ and for $\rho\in[1,+\infty)$ and $q\in[1,+\infty]$, we have
\begin{align*}
   \int_{\R\times\R}|y-x|^\rho M_{n,n^{-\alpha}}(dx,dy)\rightarrow 1
  ,\;\mathcal{W}_q(\mu_{n,n^{-\alpha}}, \nu_{n,n^{-\alpha}})\sim n^{\alpha\frac{(1-q)}{q}-\frac 1 q}
\end{align*}
and $\sigma_{\frac{q(\rho-1)}{q-1}}^{\rho-1}(\nu_{n,n^{-\alpha}})\sim \frac{n^{\rho-1}}{2^{\rho-1}\left(1+\frac{q(\rho-1)}{q-1}\right)^{\frac{q-1}{q}}}$ where $\left(1+\frac{q(\rho-1)}{q-1}\right)^{\frac{q-1}{q}}=1$ by convention  when $q=1$ 
so that
\begin{align*}
   \frac{\int_{\R\times\R}|y-x|^\rho M_{n,n^{-\alpha}}(dx,dy)}{\mathcal{W}_q(\mu_{n,n^{-\alpha}}, \nu_{n,n^{-\alpha}})\sigma_{\frac{q(\rho-1)}{q-1}}^{\rho-1}(\nu_{n,n^{-\alpha}})}\sim 2^{\rho-1}\left(1+\frac{q(\rho-1)}{q-1}\right)^{\frac{q-1}{q}}n^{\frac{q-1}{q}\alpha+\frac 1{q}+1-\rho}.
\end{align*}
Let $\rho\in[1,2)$. For $q=1$, the exponent of $n$ in the equivalent of the ratio is equal to $2-\rho>0$ so that the right-hand side goes to $+\infty$ with $n$. For $q\in (1,+\infty]$, we may choose $\alpha\in\left(\frac{q(\rho-1)-1}{q-1},1\right)$ (with left boundary equal to $\rho-1$ when $q=+\infty$) so that $\frac{q-1}{q}\alpha+\frac 1{q}+1-\rho>0$ and the right-hand side still goes to $+\infty$ with $n$. Therefore $\overline{C}_{(\rho,q),1}=+\infty$.
To prove that $\underline{C}_{(\rho,q),d}=+\infty$ for $d\ge 2$ it is enough by \cite[Lemma 1]{JoMa22} to deal with the case $d=2$, in which we use the rotation argument in \cite{BruJu22}. For $n\ge 2$ and $\theta\in(0,\pi)$, $M^\theta_n$ defined as $\frac{1}{2((n-1)n^{-\alpha}+1)}$ times
\begin{align*}
  &\delta_{((1,0),(1-n^{-\alpha}\cos\theta,-n^{-\alpha}\sin\theta))}+n^{-\alpha}\delta_{((1,0),(1+\cos\theta,\sin\theta))}+n^{-\alpha}\delta_{((n,0),(n-\cos\theta,-\sin\theta))}\\&+\delta_{((n,0),(n+n^{-\alpha}\cos\theta,n^{-\alpha}\sin\theta))}+n^{-\alpha}\sum_{i=2}^{n-1}\left(\delta_{((i,0),(i-\cos\theta,-\sin\theta))}+\delta_{((i,0),(i+\cos\theta,\sin\theta))}\right)
\end{align*}
which is a martingale coupling between the image $\mu_n$ of $\mu_{n,n^{-\alpha}}$ by $\R\ni x\mapsto (x,0)\in\R^2$ and its second marginal $\nu_n^\theta$ which, as $\theta\to 0$, converges  in any ${\cal W}_q$ with $q\in[1,+\infty]$  to the image of $\nu_{n,n^{-\alpha}}$ by the same mapping. According to the proof of \cite[Lemma 1.1]{BruJu22}, $\Pi^M(\mu_n,\nu_n^\theta)=\{M^\theta_n\}$ so that $\underline{\mathcal{M}}^\rho_\rho (\mu_{n}, \nu_{n}^\theta)=\int_{\R^2\times\R^2}|y-x|^\rho M_{n}^\theta(dx,dy)$ and 
$$\lim_{\theta\to 0}\frac{\underline{\mathcal{M}}^\rho_\rho (\mu_{n}, \nu_{n}^\theta)}{\mathcal{W}_q(\mu_{n}, \nu_{n}^\theta)\sigma_{\frac{q(\rho-1)}{q-1}}^{\rho-1}(\nu_{n}^\theta)}=\frac{\int_{\R\times\R}|y-x|^\rho M_{n,n^{-\alpha}}(dx,dy)}{\mathcal{W}_q(\mu_{n,n^{-\alpha}}, \nu_{n,n^{-\alpha}})\sigma_{\frac{q(\rho-1)}{q-1}}^{\rho-1}(\nu_{n,n^{-\alpha}})}.$$
With the above analysis of the asymptotic behaviour of the right-hand side as $n\to\infty$, we conclude that $\underline{C}_{(\rho,q),d}=+\infty$.

$(ii)$
Now, let $\rho\in[2,+\infty)$ and  $M\in\Pi^M(\mu,\nu)$.
Applying Equation \eqref{c1lemma} in Lemma \ref{c1c2lemma} for the inequality and then using the martingale property of $M$, we obtain that for $c \in \R^d$, we have
\begin{align}
\int_{\R^d\times\R^d}&\lVert x-y\rVert^\rho \,M(dx,dy) = \int_{\R^d\times\R^d}\lVert (x-c) - (y-c)\rVert^\rho \,M(dx,dy)\notag\\
&\leq \kappa_\rho \int_{\R^d\times\R^d}\left((\rho-1){\lVert x-c \rVert}^{\rho} + {\lVert y-c \rVert}^{\rho} -\rho\lVert x-c \rVert^{\rho-2}\langle x-c, y-c \rangle \right)\,M(dx,dy)\notag\\ 
&= \kappa_\rho \left(\int_{\R^d}\lVert y-c \rVert^\rho \nu(dy) - \int_{\R^d}\lVert x-c \rVert^\rho \mu(dx))\right).\label{maj1}
\end{align}
Denoting by $\pi\in \Pi(\mu,\nu)$ an optimal coupling for $\mathcal{W}_q(\mu,\nu)$, we have using Equation \eqref{c2lemma} in Lemma \ref{c1c2lemma} for the inequality
\begin{align}
   \int_{\R^d}\lVert y-c \rVert^\rho \nu(dy) &- \int_{\R^d}\lVert x-c \rVert^\rho \mu(dx)=\int_{\R^d\times\R^d}\left(\lVert y-c \rVert^\rho-\lVert x-c \rVert^\rho\right)\pi(dx,dy)\notag\\&\le {\tilde \kappa}_\rho \int_{\R^d\times\R^d} {\lVert y-x \rVert} \left( {\lVert x-c \rVert}^{\rho-1} +{\lVert y-c \rVert}^{\rho-1} \right) \, \pi(dx,dy).\label{maj2}
\end{align}
By the fact that all norms are equivalent in finite dimensional vector spaces, there exists $\lambda \in [1, \infty)$ such that for all $z \in \R^d$, we have
\[\frac{\lVert z \rVert}{\lambda} \leq \vert z\vert \leq \lambda \lVert z \rVert.\] Therefore, using \eqref{maj1} and \eqref{maj2} for the second inequality, H\"older's inequality for the fourth, the triangle inequality for the fifth and $\mu\le_{cx}\nu$ for the sixth, we get that for $c\in \R^d$,
\begin{align*}
    \int_{\R^d\times\R^d}&\vert x-y\vert^\rho \,M(dx,dy)
    \leq \lambda^\rho \int_{\R^d\times\R^d}\lVert x-y\rVert^\rho \,M(dx,dy) \\
    &\leq \kappa_\rho {\tilde \kappa}_\rho \lambda^\rho \int_{\R^d\times\R^d} {\lVert x-y \rVert} \left( {\lVert x-c \rVert}^{\rho-1} +{\lVert y-c \rVert}^{\rho-1} \right) \, \pi(dx,dy)\\
    &\le \kappa_\rho {\tilde \kappa}_\rho \lambda^{2\rho} \int_{\R^d\times\R^d} \vert x-y\vert \left( {\vert x-c \vert}^{\rho-1} +{\vert y-c \vert}^{\rho-1} \right)\, \pi(dx,dy)\\
    &\le \kappa_\rho {\tilde \kappa}_\rho \lambda^{2\rho} \mathcal W_{q}(\mu,\nu) \left(\int_{\R^d\times\R^d} \left( {\vert x-c \vert}^{\rho-1} +{\vert y-c \vert}^{\rho-1} \right)^{\frac{q}{q-1}}\, \pi(dx,dy)\right)^{\frac{q-1}{q}} \\&\le \kappa_\rho {\tilde \kappa}_\rho \lambda^{2\rho} \mathcal W_{q}(\mu,\nu) \left(\left(\int_{\R^d} \vert x-c \vert^{\frac{q(\rho-1)}{q-1}}\mu(dx)\right)^{(q-1)/q} +\left(\int_{\R^d} \vert y-c \vert^{\frac{q(\rho-1)}{q-1}}\nu(dy)\right)^{(q-1)/q}\right)\\&\le 2\kappa_\rho {\tilde \kappa}_\rho \lambda^{2\rho} \mathcal W_{q}(\mu,\nu)\left( \int_{\R^d}{\vert y -c\vert}^{\frac{q(\rho-1)}{q-1}}\nu(dy)\right)^{\frac{q-1}{q}}.
\end{align*}
By taking the infimum with respect to $c\in\R^d$, we conclude that the statement holds with $\overline{C}_{(\rho,q),d}\le 2\kappa_\rho{\tilde \kappa}_\rho\lambda^{2\rho}$. Finally, let us suppose that $\R^d$ is endowed with the Euclidean norm. Then we can choose $\lambda=1$, so that  $\overline{C}_{(\rho,q),d}\le 2\kappa_\rho{\tilde \kappa}_\rho$ with the right-hand side not depending on $d$ according to Lemma \ref{c1c2lemma}. Moreover, by Remark \ref{remrho2}, $\overline{C}_{(2,q),d}\le 2$ and since for $\alpha\in[0,1)$, \begin{equation*}
\lim_{n \to \infty}\frac{\overline{\mathcal{M}}^2_2 (\mu_{n,n^{-\alpha}}, \nu_{n,n^{-\alpha}})}{\sqrt{\mathcal{W}_1 (\mu_{n,n^{-\alpha}}, \nu_{n,n^{-\alpha}})\sigma_\infty( \nu_{n,n^{-\alpha}})}} = 2,
\end{equation*}
we have $\overline{C}_{(2,q),d}=2$.
\end{proof}

\begin{proof} [Proof of Lemma \ref{c1c2lemma}]
$ $
We suppose that $\rho>2$ since the case $\rho=2$ has been addressed in Remark \ref{remrho2}.

Suppose $x \neq 0$ and $y\neq x$ and set $e = \frac{x}{\lVert x \rVert}$ and $z=
\frac{\langle y, x\rangle}{\|x\|^2}$. The vector $\frac{y}{\lVert x \rVert}-ze$ is orthogonal to $e$ and can be rewritten as $\omega e^\perp$ with $\omega \geq 0$ and $e^\perp\in\R^d$ such that $ \lVert e^\perp \rVert=1$ and $\langle e, e^\perp\rangle = 0$. One then has $\frac{y}{\lVert x \rVert} = ze + \omega e^\perp$ and since $y\ne x$, $(z,w)\ne (1,0)$.

The first inequality \eqref{c1lemma} divided by $\lVert x\rVert^\rho$ writes:
\begin{align*}
\left((1 - z)^2 + \omega^2 \right) ^\frac{\rho}{2} &\leq \kappa_\rho \left((\rho - 1) + \left(z^2 + \omega^2 \right)^\frac{\rho}{2} -\rho z \right).
\end{align*}

Let us define $\varphi(z, \omega)=\rho -1 +(z^2 + \omega^2)^{\frac{\rho}{2}}-\rho z= - \rho(z-1) - 1  + \left(1+ 2(z-1) + (z-1)^2  +\omega^2\right)^{\frac{\rho}{2}}$ as the second factor in the right-hand side. Applying a Taylor's expansion at $t=0$ to $t\mapsto (1+t)^{\frac\rho 2}$, we obtain
\begin{equation*}
\varphi(z, \omega) = \frac{\rho}{2}\omega^2+\frac{\rho}{2}(\rho - 1)(z-1)^2 + o((z-1)^2 +\omega^2).
\end{equation*}

Since $\rho > 2$, we conclude that
\[
\lim_{(z, \omega)\to (1,0)}\frac{((1-z)^2 +\omega^2)^{\frac{\rho}{2}}}{\varphi(z, \omega)} = 0.
\]

As $\vert (z,\omega) \vert \to +\infty$, $\varphi(z,\omega) \sim (z^2 + \omega^2)^{\frac{\rho}{2}}\sim\left((z-1)^2 +\omega^2\right)^\rho$. Therefore,
\[
\lim_{\vert(z, \omega)\vert \to +\infty}\frac{((z-1)^2 +\omega^2)^{\frac{\rho}{2}}}{\varphi(z, \omega)} = 1.
\]

The function $(z,w)\mapsto\frac{((z-1)^2 +\omega^2)^{\frac{\rho}{2}}}{\varphi(z, \omega)}$ being continuous on $\R^2\setminus \{(1,0)\}$, we deduce that 

\[
1\le \sup_{(z, \omega)\neq (1,0)}\frac{((z-1)^2 +\omega^2)^{\frac{\rho}{2}}}{\varphi(z, \omega)} < +\infty.
\]
Since when $x=0$ or $y=x$, \eqref{c1lemma} holds with $\kappa_\rho$ replaced by $1$, we conclude that the optimal constant is $\kappa_\rho=\sup_{(z, \omega)\neq (1,0)}\frac{((z-1)^2 +\omega^2)^{\frac{\rho}{2}}}{\varphi(z, \omega)}$.

For the second inequality \eqref{c2lemma}, we can apply the same approach: divided by $\lVert x \rVert^\rho$, it writes

\begin{align*}
\left(z^2 + \omega^2 \right) ^\frac{\rho}{2} - 1 &\leq {\tilde \kappa}_\rho \left((z-1)^2 + \omega^2 \right) ^\frac{1}{2} \left((z^2 + \omega^2)^\frac{\rho-1}{2} + 1\right).
\end{align*}

As $(z,\omega) \to (1,0)$, $\left(z^2 + \omega^2\right)^\frac{\rho}{2} -1 = \left(1+2(z-1) +(z-1)^2 + \omega^2\right)^\frac{\rho}{2} - 1 \sim \frac{\rho}{2}\left(2(z-1) + \omega^2 \right)$

\[
\limsup_{(z, \omega) \to (1,0)}\frac{\left(z^2 + \omega^2 \right) ^\frac{\rho}{2} - 1}{\left((z-1)^2 + \omega^2 \right) ^\frac{1}{2} \left(1 +(z^2 + \omega^2)^\frac{\rho-1}{2}\right)}= \limsup_{z\to 1}\frac{\rho(z-1)}{2|z-1|} = \frac{\rho}{2}.
\]

On the other hand,
\[
\lim_{\vert(z, \omega)\vert \to +\infty}\frac{\left(z^2 + \omega^2 \right) ^\frac{\rho}{2} - 1}{\left((z-1)^2 + \omega^2 \right) ^\frac{1}{2} \left(1 +(z^2 + \omega^2)^\frac{\rho-1}{2}\right)} =1.
\]

By continuity of the considered function over $\R^2\setminus\{(1,0)\}$, we deduce that 

\[
\frac\rho 2\vee 1\le \sup_{(z, \omega)\neq (1,0)}\frac{\left(z^2 + \omega^2 \right) ^\frac{\rho}{2} - 1}{\left((z-1)^2 + \omega^2 \right) ^\frac{1}{2} \left(1 +(z^2 + \omega^2)^\frac{\rho-1}{2}\right)} < +\infty.
\]
Since when $x=0$ or $y=x$, \eqref{c2lemma} holds with ${\tilde \kappa}_\rho$ replaced by $1$, we conclude that the optimal constant is ${\tilde \kappa}_\rho=\sup_{(z, \omega)\neq (1,0)}\frac{\left(z^2 + \omega^2 \right) ^\frac{\rho}{2} - 1}{\left((z-1)^2 + \omega^2 \right) ^\frac{1}{2} \left(1 +(z^2 + \omega^2)^\frac{\rho-1}{2}\right)}$.
\end{proof}

\bibliographystyle{plain}
\bibliography{biblio.bib}

\end{document}